\documentclass[11pt]{amsart}
\usepackage{amssymb,amsmath,amsthm}
\newcommand{\leg}[2]{\genfrac{(}{)}{}{}{#1}{#2}}

\newtheorem{theorem}{Theorem}
\newtheorem{lemma}[theorem]{Lemma}
\newtheorem{corollary}[theorem]{Corollary}
\newtheorem*{conjecture}{\bf Conjecture}
\newtheorem{proposition}[theorem]{Proposition}

\theoremstyle{remark}

\newtheorem*{remark}{Remark}

\numberwithin{theorem}{section} \numberwithin{equation}{section}

\newcommand{\R}{\mathbb{R}}
\newcommand{\C}{\mathbb{C}}

\newcommand{\Z}{\mathbb{Z}}
\newcommand{\N}{\mathbb{N}}
\newcommand{\SL}{{\text {\rm SL}}}

\newcommand{\re}{\textnormal{Re}}

\begin{document}
\title[Aymptotics for rank partition functions]{Asymptotics for  rank partition functions} 

\author{Kathrin Bringmann}
\address{School of Mathematics\\University of Minnesota\\ Minneapolis, MN 55455 \\U.S.A.}
\email{bringman@math.umn.edu}    
%\dedicatory{Very preliminary}

\date{\today}
%\begin{abstract}
%\end{abstract}
\maketitle

\begin{abstract}
In this paper, we obtain asymptotic formulas for an infinite class of rank generating functions.
As an application, we solve a conjecture of Andrews and Lewis on inequalities between certain ranks.
\end{abstract}
\section{Introduction and Statement of Results}\label{intro}
A {\it partition} of a positive integer $n$ is any non-increasing
sequence of positive integers whose sum is  $n$. As usual, let
$p(n)$ denote the number of partitions of $n$. The partition
function $p(n)$ has the well known infinite product generating
function
\begin{equation}\label{partgen}
1+ 
\sum_{n=1}^{\infty}p(n)q^n=\prod_{n=1}^{\infty} \frac{1}{1-q^n}
=1+\sum_{n=1}^{\infty}
\frac{q^{n^2}}{(1-q)^2(1-q^2)^2\cdots (1-q^n)^2}.
\end{equation}
Hardy and Ramanujan showed the following asymptotic formula for $p(n)$
\begin{eqnarray*}
p(n)\sim \frac{1}{4n\sqrt{3}}\cdot e^{\pi\sqrt{2n/3}} \qquad  \qquad (n \to \infty).
\end{eqnarray*}
Using the modularity of the generating  function  for $p(n)$, Rademacher obtained an exact formula for $p(n)$.
To state his result, let $I_{s}(x)$ be the usual $I$-Bessel
function of order  $s$, and let $e(x):=e^{2\pi i x}$. Furthermore,
if $k\geq 1$  and $n$ are integers, then let
\begin{equation}\label{Akn}
A_k(n):=\frac{1}{2} \sqrt{\frac{k}{12}} \sum_{\substack{x \pmod {24k}\\
x^2 \equiv -24n+1 \pmod{24k}}}  \chi_{12}(x)
\cdot e\left(\frac{x}{12k}\right),
\end{equation}
where the sum runs over the residue classes modulo $24k$, and
where
\begin{eqnarray} \label{chi12}
\chi_{12}(x)
:= \leg{12}{x}.
\end{eqnarray}
If $n$ is a positive integer, then  Rademacher
showed that
\begin{equation}\label{Radformula}
p(n)= \frac{2 \pi}{ (24n-1)^{3/4}}
\sum_{k =1}^{\infty}
\frac{A_k(n)}{k}\cdot  I_{\frac{3}{2}}\left( \frac{\pi
\sqrt{24n-1}}{6k}\right).
\end{equation}

The partition function also satisfies some nice congruence properties; the most famous ones are the so-called Ramanujan congruences:
\begin{displaymath}
\begin{split}
p(5n+4)&\equiv 0\pmod 5,\\
p(7n+5)&\equiv 0\pmod 7,\\
p(11n+6)&\equiv 0\pmod{11}.\\
\end{split}
\end{displaymath}
In order to explain the congruences with modulus $5$ and $7$  combinatorially, 
Dyson \cite{Dy1} introduced   the ``rank" of a partition.
The rank of a partition is defined to be its
largest part minus the number of its parts. 
Dyson conjectured that the partitions of $5n+4$ (resp. $7n+5$) form $5$ (resp. $7$) groups of equal size when sorted by  their ranks modulo $5$ (resp. $7$). This conjecture was proved in 1954 by Atkin and Swinnerton-Dyer \cite{ASD}. 

If $N(m,n)$ denotes the number of partitions of $n$ with rank
$m$, then it is well known that
\begin{equation}
R(w;q):=1+\sum_{n=1}^{\infty}\sum_{m=-\infty}^{\infty}
N(m,n)w^mq^n= 1+ \sum_{n=1}^{\infty} \frac{q^{n^2}}{(wq;q)_n
(w^{-1}q;q)_n},
\end{equation}
where
$$
(a;q)_n:=(1-a)(1-aq)\cdots (1-aq^{n-1}).
$$
Obviously, by letting $w=1$, we obtain (\ref{partgen}). 
Moreover, if $N_e(n)$ (resp. $N_o(n)$) denotes the number of
partitions of $n$ with even (resp. odd) rank, then by letting
$w=-1$ we obtain
\begin{equation}\label{rankgenfcn}
1+\sum_{n=1}^{\infty}(N_e(n)-N_o(n))q^n= 1+ \sum_{n=1}^{\infty}
\frac{q^{n^2}}{(1+q)^2(1+q^2)^2\cdots (1+q^n)^2}\ .
\end{equation} 
In the following we denote this series by $f(q)$ and its $n$-th Fourier coefficient by $\alpha(n)$.
The  series  $f(q)$ is one of the third order mock theta functions defined by 
Ramanujan  in his last letter to Hardy dated January 1920
(see pages 127-131 of  \cite{Ra}). There 
 Ramanujan  claimed, without including a proof, that 
$$
\alpha(n)=(-1)^{n-1}\frac{\exp\left(\pi\sqrt{\frac{n}{6}-\frac{1}{144}}\right)}
{2\sqrt{n-\frac{1}{24}}}+O\left(
\frac{\exp\left(\frac{1}{2}\pi\sqrt{\frac{n}{6}-\frac{1}{144}}\right)}
{\sqrt{n-\frac{1}{24}}}
\right).
$$
Dragonette \cite{Dr}  proved this claim in her Ph.D. thesis 
written 1951 under the direction of Rademacher. Andrews \cite{An2}  improved this in his  Ph.D.
thesis 1964, also written under Rademacher, as 
\begin{equation}\label{andrews66}
\alpha(n)=\frac{\pi}{ \sqrt{24n-1}}
\sum_{k=1}^{[\sqrt{n}\,]}
\frac{ (-1)^{\lfloor
\frac{k+1}{2}\rfloor}A_{2k}\left(n-\frac{k(1+(-1)^k)}{4}\right)}{k}
\cdot I_{\frac{1}{2}}\left(\frac{\pi
\sqrt{24n-1}}{12k}\right)+O(n^{\epsilon}).
\end{equation}
Moreover Andrews and Dragonette made the following conjecture.
\begin{conjecture} {\text {\rm (Andrews-Dragonette)}}\newline
If $n$ is a positive integer, then
\begin{equation}\label{conj}
\alpha(n)=\frac{\pi}{(24n-1)^{\frac{1}{4}}}\sum_{k=1}^{\infty} \frac{
(-1)^{\lfloor
\frac{k+1}{2}\rfloor}A_{2k}\left(n-\frac{k(1+(-1)^k)}{4}\right)}{k}
\cdot I_{\frac{1}{2}}\left(\frac{\pi \sqrt{24n-1}}{12k}\right).
\end{equation}
\end{conjecture}
In \cite{BO1}  Ono and the author proved this conjecture using the theory of Maass-Poincar\'e series. It turns out that  $q^{-1}f\left(q^{24}\right)$  is the ``holomorphic part" of  a weak Maass form (see \cite{BO2} for the definition of a weak Maass form). In \cite{BO2} 
they showed that a similar phenomenon is true for all functions 
$$
R(\zeta_c^a;q)=:
1 + \sum_{n=1}^{\infty}A \left(\frac{a}{c};n \right)q^n,
$$
 where $\zeta_n:=e^{\frac{2 \pi i}{ n}}$ and  $0<a<c$ are  integers.
 More specificly, $R\left(\zeta_c^a;q \right)$  is the holomorphic part of a weak Maass form of weight $\frac{1}{2}$. Using this deeper insight, in this paper, 
 we are able to obtain  asymptotic formulas for all the coefficients $A\left(\frac{a}{c};n \right)$, which in turn implies asymptotics for the rank partition functions.

Before we state our result, we need some more notation. 
We let $k$ and $h$ be coprime integers, $h'$ defined by $hh' \equiv -1 \pmod k$ if $k$ is odd and $hh' \equiv -1 \pmod{2k}$ if $k$ is even, $k_1:=\frac{k}{\gcd(k,c)}$, $c_1:=\frac{c}{\gcd(k,c)}$ and $0<l<c_1$ defined by the congruence $l\equiv a k_1\pmod{c_1}$. 
If $\frac{b}{c} \in (0,1) \setminus \{ \frac{1}{2}, \frac{1}{6}, \frac{5}{6}\}$, then
define the integer $s(b,c)$ by
\begin{equation}\label{k}
s(b,c):=\begin{cases}
0 \ \ \ \ \ &{\text {\rm if}}\ 0<\frac{b}{c}<\frac{1}{6},\\
1 \ \ \ \ \ &{\text {\rm if}}\ \frac{1}{6}< \frac{b}{c}<\frac{1}{2},\\
2 \ \ \ \ \ &{\text {\rm if}}\ \frac{1}{2}< \frac{b}{c}<\frac{5}{6},\\
3 \ \ \ \ \ &{\text {\rm if}}\ \frac{5}{6}< \frac{b}{c}<1.
\end{cases}
\end{equation}
In particular, set $s:=s(l,c_1)$.
Let $\omega_{h,k}$ be the multiplier occuring in the transformation law of the partition function $p(n)$.
This is explicitly given by
\begin{eqnarray} \label{omega}
\omega_{h,k} :=
\exp\left(\pi i t(h,k) \right),
\end{eqnarray}
where 
\begin{eqnarray*}
t(h,k):= \sum_{\mu \pmod k}  \left( \left( \frac{\mu}{k}\right) \right)   \left( \left( \frac{h \mu}{k}\right) \right).
\end{eqnarray*}
Here 
\begin{eqnarray*}
((x)):= \left \{ 
\begin{array}{ll}
x- \lfloor x \rfloor - \frac{1}{2} &\text{if } x \in \R \setminus \Z ,\\
0&\text{if } x \in \Z.
\end{array}
\right.
\end{eqnarray*}
Moreover we define, for $n,m \in \Z$,  the following sums of Kloosterman type
\begin{multline} \label{kl1}
B_{a,c,k} (n,m) :=(-1)^{ak+1} \sin \left(\frac{\pi a}{c} \right)\sum_{h \pmod k^*} \frac{\omega_{h,k}
 }{\sin \left( \frac{\pi a h'}{c} \right)}  \cdot e^{-\frac{3 \pi i a^2 k_1 h'}{c}}  \cdot
e^{\frac{2 \pi i}{k}(nh+mh')}
\end{multline}
if $c|k$, 
and 
\begin{eqnarray} \label{kl2}
D_{a,c,k}(n,m) 
:= (-1)^{ak+l}
    \sum_{h \pmod k^*}
 \omega_{h,k} \cdot 
e^{\frac{2 \pi i }{k}(nh+mh')}.
\end{eqnarray}
Here the sums  run through all primitive residue classes modulo $k$.
Moreover, for $c \nmid k$, let
\begin{eqnarray} \label{del}
\delta_{c,k,r} : =
\left\{ 
\begin{array}{ll}
-\left( \frac{1}{2}+r \right)\frac{l}{c_1} 
+\frac{3}{2} \left( \frac{l}{c_1}\right)^2 +\frac{1}{24} & \text{if } 0 < \frac{l}{c_1}< \frac{1}{6},\\
-\frac{5l}{2c_1} +\frac{3}{2} \left(\frac{l}{c_1} \right)^2 +\frac{25}{24} 
- r\left(1-\frac{l}{c_1} \right)&\text{if } \frac{5}{6}<\frac{l}{c_1}<1,\\
0&\text{otherwise, }
\end{array}
\right.
\end{eqnarray}
and  for $0<\frac{l}{c_1}<\frac{1}{6}$ or  $\frac{5}{6}<\frac{l}{c_1}<1$
\begin{displaymath}
m_{a,c,k,r}:=
\left\{ 
\begin{array}{ll}
\frac{1}{2c_1^2} 
\left(-3 a^2 k_1^2+6lak_1-ak_1c_1-3l^2+lc_1-2ark_1c_1
+2 lc_1r \right)&\text{if } 0<\frac{l}{c_1}<\frac{1}{6},
\\
\frac{1}{2c_1^2} \left(- 6 ak_1c_1-3 a^2k_1^2+6lak_1 + ak_1c_1 
+ 6lc_1\right.&\text{if } \frac{5}{6}<\frac{l}{c_1}<1.
\\\left. \qquad
-3l^2 -2c_1^2-lc_1+ 2 a rk_1 c_1 + 2c_1(c_1-l)r
 \right) &
\end{array}
\right.
\end{displaymath}
\begin{remark}
It is  not hard  to see that $m_{a,c,k,r} \in \Z$.
\end{remark}
We obtain  following asymptotic formulas for the coefficients $A\left(\frac{a}{c};n \right)$.
\begin{theorem}\label{main1}
If $0<a<c$ are coprime  integers and $c$ is odd,  then   for positive integers $n$ we have that
\begin{multline*}
   A\left(\frac{a}{c};n\right) = 
     \frac{4 \sqrt{3} i }{  \sqrt{24n-1}} \sum_{1 \leq k \leq \sqrt{n} \atop c|k} 
\frac{B_{a,c,k}(-n,0)}{\sqrt{k}}  \cdot 
\sinh \left(\frac{\pi  \sqrt{24n-1} }{6k}\right)    
\\
+
 \frac{ 8 \sqrt{3}   \cdot   \sin \left(\frac{\pi a}{c} \right) }{\sqrt{24n-1}} 
    \sum_{1 \leq k\leq \sqrt{n}\atop {c \nmid k\atop {r \geq 0  \atop \delta_{c,k,r}>0}} } 
\frac{D_{a,c,k}(-n,m_{a,c,k,r})}{\sqrt{k}}   \cdot
 \sinh \left( 
 \frac{\pi\sqrt{2 \delta_{c,k,r}(24n-1)}}{\sqrt{3}k}
 \right)
 +O_c \left( n^{\epsilon}\right).
\end{multline*}
\end{theorem}

\smallskip
\noindent
\textit{Four remarks.}

\noindent 1) 
One can easily see that the second sum is empty for $c\in \{3,5\}$.

\noindent 2) 
 For a fixed choice of $a$ and $c$, we can in principle obtain exact formulas for $A \left(\frac{a}{c};n \right)$ by modifying an argument given in \cite{BO1} to prove (\ref{conj}).

\noindent 3) Similarly as in this paper, one can also prove asymptotic formulas for generalized ranks as defined in \cite{Ga}. Since the proof is basically the same, we do not give it here.

\noindent 4) One could also generalize our results to the case that $c$ is even, but for simplicity we restrict us to the case that $c$ is odd.

If we denote by   $N(a,c;n)$ the number of partitions of $n$ with rank congruent $a \pmod c$ then it is easy to conclude  the following:
\begin{corollary} \label{cor1}
For integers $0\leq a<c$, where $c$ is an odd integer,   we have
\begin{multline*}
N(a,c;n) =  
\frac{2 \pi}{ c \cdot \sqrt{24n-1}}
\sum_{k =1}^{\infty}
\frac{A_k(n)}{k}\cdot  I_{\frac{3}{2}}\left( \frac{\pi
\sqrt{24n-1}}{6k}\right) \\
+ \frac{1}{c}\sum_{j=1}^{c-1} \zeta_c^{-aj} 
\left( \frac{4 \sqrt{3} i  }{\sqrt{24n-1}} \sum_{c|k} 
\frac{B_{j,c,k}(-n,0)}{\sqrt{k}} 
\sinh \left(\frac{\pi}{6k} \sqrt{24n-1} \right) \right. \\ 
+
  \frac{8 \sqrt{3}     \sin \left(\frac{\pi j}{c} \right)}{\sqrt{24n-1}} 
   \sum_{k,r\atop {c\nmid k  \atop \delta_{c,k,r}>0} } 
\frac{D_{j,c,k}(-n,m_{j,c,k,r})}{\sqrt{k}}  
\left. \sinh \left( 
 \sqrt{\frac{2 \delta_{c,k,r}(24n-1)}{3}}
 \frac{\pi}{k} 
 \right) \right)
  + O_{c}\left(n^{\epsilon}\right).
\end{multline*}
\end{corollary}
This corollary implies some conjectures of Andrews and Lewis.  In \cite{AL,Le}  they
showed 
\begin{eqnarray*}
N(0,2;2n) & <\  N(1,2; 2n) \ \ \ \ &{\text {\rm if}}\ n\geq 1,\\
N(0,4;n)  & > \ N(2,4; n) \ \ \ \ \ \ &{\text {\rm if}}\
26<n\equiv 0,
1\pmod 4,\\
N(0,4;n)  & < \ N(2,4;n)  \ \ \ \ \ \ &{\text {\rm if}}\
26<n\equiv 2, 3\pmod 4.
\end{eqnarray*}
Moreover, they conjectured (see Conjecture 1 of \cite{AL}).
\begin{conjecture} \ {\text {\rm (Andrews and Lewis)}}\newline
For all $n>0$, we have
\begin{eqnarray} 
\begin{array}{ll} \label{conj1}
  N(0,3;n) < N(1,3;n)& \text{if } n \equiv  0 \text{ or } 2 \pmod 3, \\ \label{conj2}
  N(0,3;n) > N(1,3;n)& \text{if } n \equiv  1 \pmod 3, 
  \end{array} 
  \end{eqnarray}
  \end{conjecture} 
  A careful analysis of Corollary \ref{cor1} gives the following theorem.
  \begin{theorem} \label{main2}
The Andrews-Lewis Conjecture is true for all $n \not \in \{3,9,21 \}$ in which case we have equality in (\ref{conj1}).
\end{theorem}
\begin{remark}
From Corollary \ref{cor1} we see that 
\begin{eqnarray} \label{inspect}
N(0,3;n)-N(1,3;n)    \sim - \frac{8 \sin \left(\frac{\pi}{18} - \frac{2 \pi n}{3}\right)  \sinh \left( \frac{\pi \cdot \sqrt{24n-1}}{18}\right)}{\sqrt{24n-1}}.
\end{eqnarray}
This directly implies  the Andrews-Lewis Conjecture for $n$ sufficiently large $n$. 
For example for $n=1200$, we have 
$$
N(0,3;1200)-N(1,3;1200) 
= -2987323 8925 
$$
whereas (\ref{inspect}) gives
$$
-29873204830.34.
$$
\end{remark}
The paper is organized as follows: In Section  \ref{TRANS} we prove a transformation law for  the functions $R \left( \zeta_c^a;q\right)$. The behavior under the generators of $\SL_2(\Z)$ was  also studied in \cite{GM} and in \cite{BO2}. However here we prove  a more general result since we need the occurring roots of unity and integrals explicitly for every element in $\SL_2(\Z)$. In Section \ref{ProofTH1} we prove Theorem \ref{main1} and Corollary \ref{cor1} by using the Circle Method. For this we need some estimates shown in Section \ref{EST}. Section \ref{ProofTH2} is dedicated to the proof of Theorem \ref{main2}.
\section*{Acknowledgements}
The author thanks K. Ono, B. Kane, and F. Garvan for helpful comments to an earlier version of the paper.
\section{Modular transformation formulas} \label{TRANS}
In this section  we prove a transformation law for  the functions $R \left( \zeta_c^a;q\right)$.
For this define  
\begin{eqnarray} \label{Nfunction}
N\left (\frac{a}{c};q\right):=
\frac{1}{(q;q)_{\infty}}
\left(
1+ \sum_{n =1}^{\infty} \frac{(-1)^n \left( 1+q^n \right) \left(2- 2 \cos  \left(\frac{2 \pi a}{c} \right)\right)}
{1 - 2 q^n \cos  \left(\frac{2 \pi a}{c} \right)+q^{2n}} \cdot
q^{\frac{n(3n+1)}{2}}   \right).
\end{eqnarray}
In \cite{GM}  it is shown that 
$$R\left(\zeta_c^a;q \right)=
N\left(\frac{a}{c};q \right).
$$
Moreover let 
\begin{multline*}
N(a,b,c;q)
:=\frac{i}{2(q;q)_{\infty}}
\left(
\sum _{m=0}^{\infty} 
\frac{(-1)^{m} e^{-\frac{\pi i a}{c}}\cdot q^{\frac{m}{2}(3m+1) + ms(b,c) + \frac{b}{2c}}}{1- e^{-\frac{2 \pi i a}{c}} \cdot q^{m+\frac{b}{c}}} \right. \\
\left.
-
\sum _{m=1}^{\infty} 
\frac{(-1)^{m} e^{\frac{\pi i a}{c}}\cdot q^{\frac{m}{2}(3m+1) - ms(b,c) - \frac{b}{2c}}}{1- e^{\frac{2 \pi i a}{c}} \cdot q^{m-\frac{b}{c}}}
\right).
\end{multline*}
\begin{remark}
It is easy to see that the above definition coincides with the definition, given in \cite{BO2}.
\end{remark}
 For each $\nu \in \Z$ define
\begin{eqnarray*}
H_{a,c}(x)&:=&\frac{\cosh(x)}{\sinh\left(x+\frac{\pi i  a}{c}\right)\cdot \sinh\left(x-\frac{\pi i a}{c} \right)}, \\
I_{a,c,k,\nu}(z)
&:=& \int_{\R} e^{-\frac{3 \pi zx^2}{k}}\cdot H_{a,c}\left(\frac{\pi i \nu}{k} -\frac{\pi i }{6k}-\frac{\pi zx}{k} \right)
dx.
\end{eqnarray*}
It is easy to see that 
\begin{eqnarray*}
H_{a,c}(-x)&=&H_{a,c}(x)\\
H_{a,c}(x)&=& \frac{(1+q)\cdot q^{\frac{1}{2}}}{2\left(1- 2\cos\left(\frac{2 \pi a}{c} \right)q+q^2\right)}
\end{eqnarray*}
for $q=e^{2x}$.
We show the following transformation law for the function $N \left(\frac{a}{c};q \right)$.
\begin{theorem} \label{transfo}
We assume the same notation as in the introduction. Moreover, let
%Let $(h,k)=1$, 
$z \in \C$ with $\re(z)>0$,
%$a,c \in \N$ with $0<a<c$ and $c$ odd. Moreover define $h'$ by 
%$h h'  \equiv -1 \pmod k$ if $k$ is odd and $hh' \equiv -1 \pmod {2k}$ if $k$ is even. Furthermore let 
$q:=e^{\frac{2 \pi i }{k} (h+iz)}$, and    $q_1:=e^{\frac{2 \pi i }{k} \left(h'+\frac{i}{z}\right)}$.
%where $c_1=\frac{c}{gcd(c,k)}$, $0<l<c_1$, such that $l \equiv a k_1 \pmod{c_1}$ where
%$k_1=\frac{k}{gcd(c,k)}$, $s=s(l,c_1)$.
\begin{enumerate}
\item
If $c|k$, then 
\begin{eqnarray*}
N \left(\frac{a}{c}; q \right)
&=& 
\frac{(-1)^{ak+1} i \sin \left(\frac{\pi a}{c} \right)\cdot \omega_{h,k}}{
\sin \left(\frac{\pi a h'}{c} \right) z^{\frac{1}{2}}
}    \cdot e^{-\frac{3 \pi i a^2 k_1 h'}{c}} \cdot 
  e^{ \frac{\pi}{12k} \left( z^{-1}-z\right)} \cdot 
  N \left(\frac{ah'}{c} ;q_1\right) 
\\ 
&& +  \frac{2 \sin^2 \left(\frac{\pi a}{c} \right) \cdot    \omega_{h,k} }{k}e^{ - \frac{ \pi z}{12k}} 
\cdot z^{\frac{1}{2}}  
\sum_{ \nu \pmod k}  (-1)^{\nu} e^{-\frac{3 \pi i h' \nu^2}{k}+\frac{\pi i h'\nu}{k}} \cdot I_{a,c,k,\nu}(z).
\end{eqnarray*}
\item
If $c \nmid k$, then 
\begin{multline*}
N \left(\frac{a}{c};q \right)
= \frac{4i  (-1)^{ak+l+1}  \cdot \sin \left( \frac{\pi a }{c}\right)\omega_{h,k}}{z^{\frac{1}{2}} } \cdot 
e^{-\frac{2 \pi i h'sa}{c}-\frac{3 \pi i a^2 h' k_1}{cc_1}+\frac{6 \pi i h' l a}{cc_1}}
\cdot q_1^{\frac{sl}{c_1}-\frac{3l^2}{2c_1^2}}
\cdot e^{ \frac{\pi}{12k} \left( z^{-1}-z\right)} \\
\times
 N \left(ah' ,\frac{lc}{c_1} ,c; q_1\right) 
 + 
  \frac{2 \sin^2 \left(\frac{\pi a}{c} \right)   \omega_{h,k} }{k}e^{ - \frac{ \pi z}{12k}} 
\cdot z^{\frac{1}{2}}  
\sum_{ \nu \pmod k}  (-1)^{\nu} e^{-\frac{3 \pi i h' \nu^2}{k}+\frac{\pi i h'\nu}{k}} I_{a,c,k,\nu}(z).
\end{multline*}
\end{enumerate}
\end{theorem}
%\begin{corollary} \label{transfo2}
%Assume the same assumptions as in Theorem \ref{transfo}. 
%\begin{enumerate}
%\item
%If $3|k$, then 
%\begin{eqnarray*}
%N \left(\frac{a}{3}; q \right)
%&=& 
%\leg{h}{3} 
% i \omega_{h,k}z^{-\frac{1}{2}}
%  e^{ \frac{\pi}{12k} \left( z^{-1}-z\right)} \cdot 
%  N \left(\frac{ah'}{3} ;q_1\right) 
%\\ 
%&& +  \frac{3 \omega_{h,k} }{2k}e^{ - \frac{ \pi z}{12k}} 
%\cdot z^{\frac{1}{2}}  
%\sum_{ \nu \pmod k}  (-1)^{\nu} e^{-\frac{3 \pi i h' \nu^2}{k}+\frac{\pi i h'\nu}{k}} I_{1,3,k,\nu}(z).
%\end{eqnarray*}
%\item
%If $3 \nmid k$, then 
%\begin{multline*}
%N \left(\frac{a}{c};q \right)
%=  
%\leg{k}{3} (-1)^{a(1+k)+1} 2 \sqrt{3}
%  i   \omega_{h,k} z^{-\frac{1}{2}} 
%e^{-\frac{ \pi i a^2 h' k}{3}}
%\cdot q_1^{\frac{l^2}{6}}
%\cdot e^{ \frac{\pi}{12k} \left( z^{-1}-z\right)} \\
%\times
% N \left(ah' ,l,3; q_1\right) 
% + 
%  \frac{3 \omega_{h,k} }{2k}e^{ - \frac{ \pi z}{12k}} 
%\cdot z^{\frac{1}{2}}  
%\sum_{ \nu \pmod k}  (-1)^{\nu} e^{-\frac{3 \pi i h' \nu^2}{k}+\frac{\pi i h'\nu}{k}} I_{1,3,k,\nu}(z).
%\end{multline*}
%\end{enumerate}
%\end{corollary}
\begin{proof}
We modify the proof of \cite{An2}.   
We easily see, using $1-\cos (2x)= 2 \sin(x)^2$ that 
\begin{eqnarray}  \label{formel1}
(q;q)_{\infty}\cdot  N \left( \frac{a}{c};q\right)
=  \sin ^2 \left(\frac{\pi a }{c} \right) 
\sum_{n \in \Z} (-1)^n \ H_{a,c}\left(\frac{\pi i n}{k} (h+iz)\right) \cdot 
 e^{   \frac{3 \pi i (h+iz)n^2}{k}}.
\end{eqnarray}
Writing $n=km+\nu$ with $0 \leq \nu <k$, $m \in \Z$ and using that $(h,k)=1$, gives that (\ref{formel1}) equals
\begin{equation}\label{formel2}
\begin{split}
 \sin ^2 \left(\frac{\pi a }{c} \right) 
 \sum_{\nu=0}^{k-1} (-1)^{\nu} e^{\frac{3 \pi i h \nu^2}{k}}
\sum_{m \in \Z}  (-1)^m \ H_{a,c} \left(\frac{\pi i h \nu}{k} -\frac{\pi (km+ \nu)z}{k}\right) \cdot  
e^{ -  \frac{3 \pi z(km+\nu)^2}{k}}.
\end{split}
\end{equation}
Using Poisson summation and substituting $x \mapsto kx+ \nu$ gives that the inner sum equals 
\begin{eqnarray} \label{formel3a}
\frac{1}{k} \sum_{n \in \Z} \int_{\R} 
 H_{a,c} \left(\frac{\pi i h\nu}{k} -\frac{\pi x z}{k}\right)\cdot 
 e^{  \frac{\pi i(2n+1)(x- \nu) }{k}-  \frac{3 \pi z  x ^2}{k}}
 dx.
\end{eqnarray}
Strictly speaking for $c|k$ there may be a pole at $x=0$. In this case we take the principal part of the integral.
Inserting  (\ref{formel3a})  into (\ref{formel2}) we  see that the summation only depends  on $\nu \pmod k$.
Moreover, by changing $\nu$ into $-\nu$,  $x$ into $-x$, and $n$ into $-(n+1)$, we see that the sum over $n$ with $n \leq -1$ equals the sum with $n \geq 0$. Thus (\ref{formel2}) equals 
\begin{multline}\label{formel3}
\frac{2  \sin ^2 \left(\frac{\pi a }{c} \right) }{k}
 \sum_{\nu \pmod k} 
 (-1)^{\nu} e^{\frac{3 \pi i h \nu^2}{k}}
 \sum_{n \in \N} \int_{\R}  
 H_{a,c} \left(\frac{\pi i h\nu}{k} -\frac{\pi x z}{k}\right)\cdot 
 e^{\frac{   \pi i (2n+1)(x- \nu)}{k}-   \frac{3 \pi z x ^2}{k}}
  dx.
\end{multline}
To see where the poles of the integrant lie, we introduce the function  
\begin{eqnarray*}
S_{a,c,k}( x):= 
\frac{\sinh(c_1x)}{\sinh\left(\frac{x}{k} + \frac{\pi i a}{c}\right)  \cdot     \sinh\left(\frac{x}{k} - \frac{\pi i a}{c}\right)}
\end{eqnarray*}
which 
is entire as  a  function of $x$. Using that $c$ is odd we can write  the integrand in (\ref{formel3})  as 
\begin{eqnarray*}
  \frac{(-1)^{h \nu} 
 \cosh \left(\frac{\pi i h\nu}{k} -\frac{\pi x z}{k}\right)\cdot 
 e^{ \frac{  \pi i (2n+1) (x- \nu)}{k}-  \frac{3 \pi z x ^2}{k}     }
 \cdot  S_{a,c,k}\left( \pi x z - \pi i h \nu\right)}{ 
 \sinh (\pi c_1x z) 
  }.
  \end{eqnarray*}
  From this we see that the only  poles can lie in the points
$$
x_m:=\frac{im}{c_1 z} \qquad (m \in \Z).
$$
We treat   the cases with $c|k$ or $c \nmid k$ seperately.

If $c |k$, then $c_1=1$.  
 One   computes that  each choice $\pm $ leads at most for one $\nu \pmod k$ to a non-zero residue, and that this $\nu$ can for $\epsilon \in \{\pm\}$ be chosen as
\begin{eqnarray*}
\nu_m^{\epsilon} :=- h' (m -\epsilon ak_1).
\end{eqnarray*}
We denote the corresponding residues by $\lambda_{n,m}^{\epsilon}$ ($\epsilon \in \{\pm \}$). By shifting the path of integration through the points 
\begin{eqnarray*}
\omega_n:=\frac{(2n+1)i}{6z},
\end{eqnarray*}
we have to take those points $x_m$ into account for which $n \geq 3m\geq 0$.
Setting $r_0:=\frac{1}{2}$ and $r_m:=1$ for $m \in \N$, we obtain by the Residue Theorem
\begin{eqnarray*}
(q;q)_{\infty} \cdot N \left( \frac{a}{c};q\right)=
\sum_{11}+\sum_{2},
\end{eqnarray*}
where
\begin{eqnarray*}
\sum_{11}&:=&
\frac{4 \pi i  \sin ^2 \left(\frac{\pi a }{c} \right) }{k}
  \sum_{m\geq 0 \atop \epsilon \in \{\pm \}} 
  r_m   (-1)^{\nu_m^{\epsilon}}\ e^{\frac{3 \pi i h(\nu_m^{\epsilon})^2}{k}}
 \sum_{n=3m}^{\infty} \lambda_{n,m}^{\epsilon},
 \\
 \sum_2&:=&
 \frac{2  \sin ^2 \left(\frac{\pi a }{c} \right) }{k}
  \sum_{\nu \pmod k}(-1)^{\nu} \ e^{\frac{3 \pi i h \nu^2}{k}}\\&&
 \sum_{n \in \N} \int_{-\infty+\omega_n}^{\infty+\omega_n}
 H_{a,c} \left(\frac{\pi i h\nu}{k} -\frac{\pi x z}{k}\right)\cdot 
 e^{\frac{\pi i (2n+1)(x- \nu)}{k}-   \frac{3 \pi x ^2z}{k}} \
  dx.
\end{eqnarray*} 
If $c \nmid k$, then  a pole can only occur   if 
$
m \equiv \pm a k_1 \pmod{c_1}
$.
Writing 
$
 c_1 m \pm l
$
instead of $m$
with $m \geq 0$ for the choice $+$ and $m >0$ for the choice $-$ and $l$ as in the introduction,
we see that to each  choice there corresponds exactly one $\nu \pmod k$ and we can choose $\nu$ for $\epsilon \in \{ \pm 1\} $ as
\begin{eqnarray*}
\nu_m^{\epsilon} :=   -h' \left( m \, \epsilon\,  \frac{1}{c_1}\left(l-a k_1\right)\right).
\end{eqnarray*}      
As before, we denote the corresponding residues by $\lambda_{n,m}^{\epsilon}$. By shifting the path of integration through the points $\omega_n$, we have to take those points $x_m$ into account for which $\frac{2n+1}{6}> \frac{c_1 m \epsilon l}{c_1}$. One can  see that this is equivalent to   
\begin{eqnarray*}
n \geq 3m \epsilon s,
\end{eqnarray*}
where $s$ was defined in (\ref{k}).
By the Residue Theorem we obtain
\begin{eqnarray*}
(q;q)_{\infty}  \cdot N \left( \frac{a}{c};q\right)=
\sum_{12}+\sum_{2},
\end{eqnarray*}
where $\sum_2$  is given as before and  $\sum_{12}$  is defined as
\begin{eqnarray*}
\frac{4 \pi i  \sin ^2 \left(\frac{\pi a }{c} \right) }{k}
\left(
  \sum_{m\geq 0} 
 (-1)^{\nu_m^{+}} e^{\frac{3 \pi i h(\nu_m^{+})^2}{k}}
 \sum_{n=3m+s}^{\infty} \lambda_{n,m}^{+}
 + 
   \sum_{m\geq 1} 
 (-1)^{\nu_m^{-}} e^{\frac{3 \pi i h(\nu_m^{-})^2}{k}}
 \sum_{n=3m-s}^{\infty} \lambda_{n,m}^{-}
 \right).
 \end{eqnarray*}    
We first consider the sums $\sum_{11}$ and $\sum_{12}$. 
We have
\begin{eqnarray*}
\lambda_{n,m}^{\epsilon}=
- \frac{k \cdot 
\cosh \left(\frac{\pi i h\nu_m^{\epsilon}}{k} -\frac{\pi x_mz}{k}\right)\cdot  
e^{\frac{ \pi i (2n+1)(x_m- \nu_m^{\epsilon})}{k} - \frac{3 \pi z x_m^2}{k}}
 }{ \pi z \cdot
  \cosh \left(
\frac{\pi i h \nu_m^{\epsilon}}{k} -\frac{\pi x_m z}{k}
\epsilon \frac{\pi i a}{c}\right)  \cdot
\sinh \left(
\frac{\pi i h\nu_m^{\epsilon} }{k} -\frac{\pi x_m z}{k}
-\epsilon  \frac{\pi i a}{c}\right)
  }.
\end{eqnarray*}
From this one directly sees that
\begin{eqnarray*}
\lambda_{n+1,m}^{\epsilon}
=\exp \left( \frac{2 \pi i}{k}\left(x_m - \nu_{m}^{\epsilon} \right)\right) \cdot
\lambda_{n,m}^{\epsilon}.
\end{eqnarray*}
Thus 
\begin{eqnarray*}
 \sum_{n=3m+r}^{\infty} \lambda_{n,m}^{\epsilon}
 =\frac{\lambda_{3m+r,m}^{\epsilon}}{1-    \exp \left( \frac{2 \pi i}{k}\left(x_m - \nu_{m}^{\epsilon} \right)\right)}
 \qquad (r \in \{0,\pm s\}, \epsilon \in \{ \pm\} ).
\end{eqnarray*}
A straightforward but lenghty calculation gives 
\begin{eqnarray*}
\sum _{11} &=&
\frac{(-1)^{ak+1}i \sin \left(\frac{\pi a }{c}\right)}{\sin\left(\frac{\pi ah'}{c} \right) z} \cdot     e^{-\frac{3 \pi i a^2 k_1h'}{c}} \cdot
(q_1;q_1)_{\infty} \cdot N\left(\frac{a h'}{c};q_1 \right),\\
\sum _{12} &=& 
\frac{4i \sin \left(\frac{\pi a }{c}\right)(-1)^{ak+l+1}}{z}\
e^{-\frac{2 \pi i  h's a }{c} - \frac{3 \pi i h' a^2 k_1}{c c_1} + \frac{6 \pi i h' l a}{cc_1} }
q_1^{ \frac{sl}{c_1}-\frac{3}{2}\left(\frac{l}{c_1} \right)^2} 
(q_1;q_1)_{\infty} N\left(h'a,\frac{lc}{c_1},c;q_1 \right).
\end{eqnarray*}
Here we have to be a little careful since for $c|k$ ($c \nmid k$) we only have that $h \nu $ is congruent to $ m -\epsilon a k_1$  ($m \epsilon \frac{1}{c_1}(l-ak_1)$) modulo $k$ but not necessarily modulo $2k$ if $k$ is odd.

We next turn to the computation of $\sum_2$.
With the same argument as before we can change the sum over $\N$ into a sum over $\Z$. Making the translation 
$x \mapsto x+ \omega_n$ and writing $n=3 p + \delta$ with $ p \in \Z$ and $ \delta \in \{ 0, \pm 1\}$ gives that 
\begin{multline*}
 \sum_2 =
 \frac{ \sin ^2 \left(\frac{\pi a }{c} \right) }{k}
  \sum_{\nu \pmod k}
  (-1)^{\nu} \ e^{\frac{3 \pi i h \nu^2}{k}}
 \sum_{p \in \Z\atop \delta \in \{ 0, \pm 1\}}
 e^{- \frac{\pi(6p+2 \delta+1)^2}{12kz}- \frac{\pi i \nu (6p+ 2 \delta+1)}{k}}
 \\
  \int_{-\infty}^{\infty}
 H_{a,c} \left(\frac{\pi i h\nu}{k} -    \frac{\pi i( 6 p+ 2\delta+1 )}{6k}       -\frac{\pi zx}{k}\right)
  \cdot     e^{-\frac{3 \pi z x^2}{k}} \ dx.
\end{multline*}
Now  $(h',k)=1$  implies that $- h'(\nu+p)$ runs modulo $k$ if $\nu$ does. Thus we can change $\nu$ into  $- h'(\nu+p)$ which leads to 
\begin{multline} \label{change}
\sum_2=
 \frac{ \sin ^2 \left(\frac{\pi a }{c} \right) }{k}
  \sum_{\nu \pmod k \atop{ \delta \in \{0,\pm 1 \}\atop p \in\Z }}
  (-1)^{\nu+p} \
  q_1^{\frac{p}{2}(3p+2 \delta+1)} \cdot 
 e^{-\frac{\pi(2 \delta+1)^2}{12 kz} +\frac{\pi i h'\left(-3 \nu^2+(2 \delta+1)\nu \right)}{k}} 
  \\
  \int_{-\infty}^{\infty}
  H_{a,c} \left(\frac{\pi i \nu}{k} -    \frac{\pi i \left(2 \delta +1\right)  }{6k}     -\frac{\pi zx}{k}\right)
  \cdot     e^{-\frac{3 \pi z x^2}{k}} \ dx.
\end{multline}
Now the integral is independent of $p$  and the sum over $p$ equals
\begin{eqnarray}\label{sump}
\sum_{p \in \Z} 
(-1)^p \cdot q_1^{\frac{1}{2}p(3p+(2 \delta+1))} . 
\end{eqnarray}
If $\delta=1$, then (\ref{sump}) vanishes since  the $p-$th and the $-(p+1)$-th term  cancel. 
Changing $\nu$ into  $- \nu$, $p$ into $-p$ and $x$ into $-x$, we see that the terms in (\ref{change}) corresponding to $\delta=-1$ and $\delta=0$  are equal. Moreover  in this case (\ref{sump}) equals
$(q_1;q_1)_{\infty}$. Thus 
\begin{eqnarray*}\
\sum_2=
\frac{2 \sin^2\left( \frac{\pi a}{c}\right) \cdot (q_1;q_1)_{\infty} \cdot e^{-\frac{\pi}{12kz}}}{k} 
\sum_{\nu \pmod k} 
(-1)^{\nu} e^{-\frac{3 \pi i h'\nu^2}{k} +\frac{\pi i h'\nu}{k}} \cdot 
I_{a,c,k,\nu}(z).
\end{eqnarray*}
Now the theorem follows easily using  the transformation law
\begin{eqnarray*}
(q_1;q_1)_{\infty} 
=\omega_{h,k} \cdot z^{\frac{1}{2}}  \cdot
e^{\frac{\pi}{12k}\left(z^{-1}-z \right) } \cdot
(q;q)_{\infty}.
\end{eqnarray*}
\end{proof}
\section{Some estimates} \label{EST}
In this section we estimate the function $I_{a,c,k,\nu}(z)$, defined in Section \ref{TRANS}, and some Kloosterman sums.
\begin{lemma} \label{intest}
Assume that $n \in \N,\, \nu \in \Z$, $z:=\frac{k}{n}-k \Phi i$, $-\frac{1}{k(k+k_1)}\leq \Phi \leq \frac{1}{k(k+k_2)}$, where $\frac{h_1}{k_1}  < \frac{h}{k}  <\frac{h_2}{k_2} $ are adjacent Farey fractions in the Farey sequence of order $N$, with $N:=\lfloor n^{\frac{1}{2}}\rfloor$. Then
\begin{eqnarray*}
z^{\frac{1}{2}} \cdot  I_{a,c,k,\nu}(z)  \ll 
k\cdot n^{\frac{1}{4}} \cdot 
g_{a,c,k,\nu},
\end{eqnarray*}
where $ g_{a,c,k,\nu},:= 
\left(\min
 \left(6 kc \left\{ \frac{\nu}{k} -\frac{1}{6k}+\frac{a}{c}\right\} ,  6 kc \left\{ \frac{\nu}{k} -\frac{1}{6k}-\frac{a}{c}\right\}     \right)
 \right)^{-1}$, where $\{x \}:=x-\lfloor x \rfloor$ for $x \in \R$. Here the implied constant is independent of $a, k$, and $\nu$.
\end{lemma}
\begin{proof} 
We write $\frac{\pi z}{k}=Ce^{iA}$ with $C>0$. Then $|A|<\frac{\pi}{2}$ since $\re(z)>0$. 
Making the substitution $\tau=\frac{\pi z x}{k}$ 
gives 
\begin{eqnarray}  \label{integralest}
z^{\frac{1}{2}} \cdot  I_{a,c,k,\nu}(z)= 
\frac{k}{\pi z^{\frac{1}{2}}}
\int_{S} e^{-\frac{3 k \tau^2}{\pi z}}\cdot H_{a,c}\left(\frac{\pi i \nu}{k} -\frac{\pi i }{6k}-\tau \right)
d\tau,
\end{eqnarray}
where  $\tau$ runs on the ray through $0$ of elements with argument $\pm A$.
One can see that  for  $0 \leq t \leq A$
\begin{eqnarray} \label{goestozero}
\left| e^{-\frac{3 k R^2 e^{2 i t}}{\pi z}}\cdot H_{a,c}\left(\frac{\pi i \nu}{k} -\frac{\pi i }{6k}\pm Re^{it} \right)
\right| \to 0  \qquad (R \to \infty).
\end{eqnarray}
Moreover, since $c$ is odd,  the integrant in (\ref{integralest})   can only have poles in points $ir$ with $r \in \R \setminus\{0 \}$. 
Thus we can  shift   the path  of integration to the real line and get
\begin{eqnarray*}
z^{\frac{1}{2}}\cdot  I_{a,c,k,\nu}(z)= 
\frac{k}{\pi z^{\frac{1}{2}}}
\int_{\R} e^{-\frac{3 k t^2}{\pi z}}\cdot H_{a,c}\left(\frac{\pi i \nu}{k} -\frac{\pi i }{6k}- t \right)
dt.
\end{eqnarray*}
We can show the following estimates 
\begin{eqnarray*}
\left|  e^{-\frac{3 k t^2}{\pi z}}\right| = e^{-\frac{3 k}{\pi} \re \left(\frac{1}{z} \right)t^2},\\
\left| \cosh \left(\frac{\pi i \nu}{k} -\frac{\pi i }{6k}- t \right)  \right|\leq e^t,
\end{eqnarray*}
\begin{displaymath}
\left| \sinh \left(\frac{\pi i \nu}{k} -\frac{\pi i }{6k}- t \pm \frac{\pi i a }{c} \right)  \right|
\geq 
\left\{
\begin{array}{ll}
\frac{e^t}{2 \sqrt{2}}&\text{if } t \geq 1,\\[1ex]
\left| \sin \left(\frac{\pi \nu}{k} -\frac{\pi  }{6k} \pm \frac{\pi  a }{c} \right)  \right|&\text{if } t \leq 1,
\end{array}
\right.
\end{displaymath}
\begin{multline*}
 \left| \sin \left(\frac{\pi \nu}{k} -\frac{\pi  }{6k} + \frac{\pi  a }{c} \right)  \right|
\left| \sin \left(\frac{\pi \nu}{k} -\frac{\pi  }{6k} - \frac{\pi  a }{c} \right)  \right| \\
\gg \left(\min
 \left( \left\{ \frac{\nu}{k} -\frac{1}{6k}+\frac{a}{c}\right\} ,    \left\{ \frac{\nu}{k} -\frac{1}{6k}-\frac{a}{c}\right\}     \right)
 \right).
\end{multline*}
Thus 
\begin{eqnarray*}
z^{\frac{1}{2}} \cdot  I_{a,c,k,\nu}(z)
\ll \frac{k}{ \left(\min
 \left( \left\{ \frac{\nu}{k} -\frac{1}{6k}+\frac{a}{c}\right\} ,    \left\{ \frac{\nu}{k} -\frac{1}{6k}-\frac{a}{c}\right\}     \right)
 \right)   |z|^{\frac{1}{2}}  }    \int_{\R} e^{-\frac{3k}{\pi}  t^2 \re\left(\frac{1}{z} \right) } dt.
\end{eqnarray*}
Making the substitution $t \mapsto \sqrt{\frac{3 k \re\left( \frac{1}{z}\right)}{\pi} } \cdot  t$ and using the estimate 
\begin{eqnarray*}
\re\left( \frac{1}{z}\right)^{-\frac{1}{2}} \cdot |z|^{-\frac{1}{2}} \leq 2^{\frac{1}{4}}  \cdot n^{\frac{1}{4}} \cdot k^{-\frac{1}{2}}
\end{eqnarray*}
gives the lemma.
\end{proof}
We next  estimate  certain  sums of Kloosterman type.
\begin{lemma} \label{kloost}
Let $n,m \in \Z$, $0 \leq \sigma_1 < \sigma_2 \leq k$, $D \in \Z$ with $(D,k)=1$. 
\begin{enumerate}
\item
We have
\begin{eqnarray} \label{kloost1}
\sum_{h\pmod k^* \atop \sigma_1 \leq Dh' \leq \sigma_2} 
\omega_{h,k} \cdot  e^{\frac{2 \pi i }{k} (hn+h'm )} 
&\ll& \gcd(24n+1,k)^{\frac{1}{2}}  \cdot 
k^{\frac{1}{2}+\epsilon}.
\end{eqnarray}
\item
If $c|k$, then we have 
\begin{multline} 
\label{kloost2}
(-1)^{ak+1} \sin\left(\frac{\pi a}{c} \right)
\sum_{h\pmod k^* \atop \sigma_1 \leq D h' \leq \sigma_2} 
\frac{\omega_{h,k}}{\sin\left(\frac{\pi a h'}{c} \right)} \cdot  e^{-\frac{3 \pi i a^2 k_1h'}{c}} \cdot e^{\frac{2 \pi i }{k} (hn+h'm )}  
\ll \gcd(24n+1,k)^{\frac{1}{2}}  \cdot
k^{\frac{1}{2}+\epsilon},
\end{multline}
\end{enumerate}
where the implied constants are independent of $a$ and $k$.
\end{lemma}
\begin{proof}
Equation  (\ref{kloost1}) is basically  proven in \cite{An2}, thus we only consider (\ref{kloost2}).
We set $\tilde c:=c$ if $k$ is odd and $\tilde c:=2c$ if $k$ is even.
Clearly
$
\frac{e^{- \frac{3 \pi i a^2k_1 h'}{c}}}{ \sin\left(\frac{\pi ah'}{c} \right)}
$
only depends on the residue class of $h' \pmod{\tilde c}$.
Thus we can write (\ref{kloost2}) as 
\begin{eqnarray*}
(-1)^{ak+1} \sin\left(\frac{\pi a}{c} \right)
\sum_{c_j}  \frac{e^{- \frac{3 \pi i a^2k_1 c_j}{c}}}{ \sin\left(\frac{\pi ac_j}{c} \right)}
\sum_{h\pmod k^*\atop {\sigma_1 \leq D h' \leq \sigma_2 \atop h' \equiv c_j \pmod{\tilde c}}} 
\omega_{h,k} \cdot
e^{\frac{2 \pi i }{k} (hn+h'm )},  
\end{eqnarray*}
where $c_j$ runs through a set of primitive residues $\pmod{\tilde c}$.
The inner sum can be rewritten as 
\begin{multline*}
\frac{1}{\tilde{c}}
\sum_{h\pmod k^* \atop \sigma_1 \leq D h' \leq \sigma_2} 
\omega_{h,k}  \cdot 
e^{\frac{2 \pi i  }{k} (hn+h'm )} 
\sum_{r \pmod{\tilde{c}}}
e^{\frac{2 \pi i r}{\tilde c}(h'-c_j)}  \\
=
\frac{1}{\tilde{c}}
\sum_{r \pmod{\tilde{c}}}
e^{-\frac{2 \pi i r c_j}{\tilde c}}
\sum_{h\pmod k^* \atop \sigma_1 \leq D h' \leq \sigma_2} 
\omega_{h,k}  \cdot 
e^{\frac{2 \pi i }{k} \left(hn+h'\left(m +\frac{kr}{\tilde c}\right)\right)} .
\end{multline*}
Using (\ref{kloost1}) now easily gives (\ref{kloost2}).
\end{proof}
\begin{remark}
If $a=1$ and $c=3$, then it is not hard to see that the left hand side of (\ref{kloost2}) simplifies to 
\begin{eqnarray*}
\sum_{h\pmod k^* \atop \sigma_1 \leq D h' \leq \sigma_2} 
\leg{h}{3} \cdot  \omega_{h,k} \cdot e^{\frac{2 \pi i }{k} (hn+h'm )}.  
\end{eqnarray*}
\end{remark}
\section{Proof of Theorem \ref{main1}} \label{ProofTH1}
\begin{proof}[Proof of Theorem \ref{main1}] 
For the proof of Theorem \ref{main1} we use the Hardy-Ramanujan method.
By Cauchy's Theorem we have  for $n>0$
\begin{eqnarray*}
A \left(\frac{a}{c};n \right)
=\frac{1}{ 2 \pi i }  
\int_{C} \frac{N \left( \frac{a}{c};q \right)}{q^{n+1}} \ dq,
\end{eqnarray*}
where $C$ is an arbitrary path inside the unit  circle surrounding $0$ counterclockwise.
Choosing the circle with radius $e^{-\frac{2\pi}{n}}$ and as a parametrisation $q= e^{-\frac{2 \pi}{n} + 2 \pi i t }$ with $0 \leq t \leq 1$,  gives
\begin{eqnarray*}
A \left(\frac{a}{c};n \right)
= \int_{0}^{1} 
N \left( \frac{a}{c}; e^{-\frac{2 \pi}{n} + 2 \pi i t}\right) \cdot e^{2 \pi- 2 \pi i n t}  \ dt. 
\end{eqnarray*} 
Define \begin{eqnarray*}
\vartheta_{h,k}' :=\frac{1}{k(k_1+k)},\quad 
\vartheta_{h,k}'' :=\frac{1}{k(k_2+k)}, 
\end{eqnarray*}
where $\frac{h_1}{k_1}  < \frac{h}{k}  <\frac{h_2}{k_2} $ are adjacent Farey fractions in the Farey sequence of order $N:=\left\lfloor n^{1/2} \right \rfloor$.  From the theory of Farey fractions it is known that 
\begin{eqnarray} \label{farey}
\frac{1}{k+k_j} \leq \frac{1}{N+1} \qquad (j =1,2).
\end{eqnarray}
We decompose the path of integration in paths along the Farey arcs 
$-\vartheta_{h,k}' \leq \Phi \leq \vartheta_{h,k}''$, where $\Phi=t-\frac{h}{k}$ and  $0 \leq h \leq k\leq N$ with $(h,k)=1$.
Thus  
\begin{eqnarray*}
A \left(\frac{a}{c};n \right) = 
\sum_{h,k} 
e^{- \frac{2 \pi i hn}{k}}
\int_{-\vartheta_{h,k}'}^{\vartheta_{h,k}''}
N\left(\frac{a}{c}; e^{\frac{2 \pi i }{k}(h+iz) } \right)  \cdot 
e^{\frac{2 \pi n z}{k}} \ d\Phi,
\end{eqnarray*}
where $z=\frac{k}{n}- k \Phi i$.
Applying Theorem \ref{transfo} 
gives
\begin{multline*} 
A \left(\frac{a}{c};n \right) 
=
  i \sin \left(\frac{\pi a}{c} \right)   \sum_{h,k\atop c|k} 
  \omega_{h,k} \ \frac{(-1)^{ak+1}}{
\sin \left(\frac{\pi a h'}{c} \right)} \cdot e^{- \frac{3 \pi i a^2 k_1h'}{c} - \frac{2 \pi i hn}{k} } 
\int_{- \vartheta_{h,k}'}^{\vartheta_{h,k}^{''}}
z^{-\frac{1}{2}} \cdot 
e^{\frac{2 \pi z}{k}\left(n-\frac{1}{24}  \right)+\frac{\pi}{12k z }}  \\
 \times
  N \left(\frac{ah'}{c} ;q_1\right)  
d \Phi 
-4    i \sin \left(\frac{\pi a}{c} \right)   \sum_{h,k\atop c\nmid k} 
  \omega_{h,k} \ (-1)^{ak+l} \
e^{-\frac{2 \pi i h' sa}{c} -\frac{3 \pi i h'a^2 k_1}{cc_1}  +\frac{6 \pi i h' l a}{cc_1}  - \frac{2 \pi i hn}{k} }  
\\
\int_{- \vartheta_{h,k}'}^{\vartheta_{h,k}^{''}}
z^{-\frac{1}{2}} \cdot
e^{\frac{2 \pi z}{k}\left(n-\frac{1}{24}  \right)+\frac{\pi}{12k z }}
 \cdot 
  q_1^{\frac{sl}{c_1} -\frac{3 l^2}{2 c_1^2} }  \cdot 
  N \left(ah',\frac{lc}{c_1},c ;q_1\right)   
d  \Phi
+  2 \sin^2 \left(\frac{\pi a}{c} \right)   
\sum_{h,k} 
\frac{\omega_{h,k}}{k} \cdot e^{-\frac{2 \pi i hn}{k}} \\
\sum_{ \nu \pmod k}  (-1)^{\nu} \ e^{-\frac{3 \pi i h' \nu^2}{k}+\frac{\pi i h'\nu}{k}} 
\int_{- \vartheta_{h,k}'}^{\vartheta_{h,k}^{''}}
e^{  \frac{ 2 \pi z}{k}\left(n-\frac{1}{24} \right)} 
\cdot z^{\frac{1}{2}}  \cdot 
 I_{a,c,k,\nu}(z) d \Phi 
 =: \sum_{1}+\sum_2 +\sum_3.
\end{multline*}   
To estimate  $\sum_1$, we  write 
\begin{eqnarray*} 
N\left(\frac{ah'}{c} ;q_1\right)=:
1+ 
\sum_{r \in \N} a(r) \cdot e^{\frac{2 \pi i m_r h'}{k}} \cdot e^{- \frac{2 \pi r}{kz}},
\end{eqnarray*} 
where $m_r$ is a  sequence in $\Z$  and the coefficients $a(r)$ are independent of  $a,c,k$, and $h$.
We treat the constant term and the term coming from from $r \geq 1$ seperately since they contribute to the main term and to the  error term, respectively.  We denote the associated sums by $S_1$ and $S_2$, respectively and first estimate $S_2$. Throughout we need the easily verified fact  that  $\re(z)=\frac{k}{n}$, $\re\left( \frac{1}{z}\right)> \frac{k}{2}$,
$|z|^{-\frac{1}{2}} \leq n^{\frac{1}{2}} \cdot k^{-\frac{1}{2}} $, and $\vartheta_{h,k}'+\vartheta_{h,k}^{''} \leq
\frac{2}{k(N+1)}$. 
Since $k_1, k_2 \leq N$, we can  write 
\begin{eqnarray} \label{pathint}
\int_{- \vartheta_{h,k}'}^{\vartheta_{h,k}^{''}}
=
\int_{-\frac{1}{k(N+k)}}^{\frac{1}{k(N+k)}}  
+ \int_{-\frac{1}{k(k_1+k)}}^{-\frac{1}{k(N+k)}} 
+ \int_{\frac{1}{k(N+k)}}^{\frac{1}{k(k_2+k)}} 
\end{eqnarray}
and denote the associated sums by $S_{21}$, $S_{22}$, and $S_{23}$, respectively.

We first consider $S_{21}$. Using Lemma \ref{kloost} gives 
\begin{multline*}
S_{21} \ll
\left| \sum_{r=1}^{\infty} 
 a(r)
 \sum_{c|k} 
 (-1)^{ak+1}\ \sin\left(\frac{\pi a}{c} \right)\sum_{h} \frac{\omega_{h,k}}{\sin\left(\frac{\pi a h'}{c} \right)}
 \cdot 
e^{- \frac{3 \pi i a^2 k_1h'}{c} - \frac{2 \pi i hn}{k} + \frac{2 \pi i m_r h' }{k}} \right.  \\ \left.
\int_{- \frac{1}{k(N+k)}}^{   \frac{1}{k(N+k)}}
z^{-\frac{1}{2}} \cdot 
e^{-\frac{2 \pi}{kz}\left(r-\frac{1}{24}\right)+\frac{2 \pi z}{k}\left(n-\frac{1}{24} \right)}  \ d \Phi  \right|
\ll 
\sum_{r=1}^{\infty} |a(r)| \cdot e^{-\pi r} 
\sum_{k} k^{-1+\epsilon} \cdot (24n-1,k)^{\frac{1}{2}} \\
\ll  \sum_{d|(24n-1)\atop d \leq N} d^{\frac{1}{2}}  
\sum_{k \leq \frac{N}{d}}  (dk)^{-1 + \epsilon} 
\ll n^{\epsilon}  \sum_{d|(24n-1)\atop d \leq N} d^{-\frac{1}{2}}  
\ll n^{\epsilon}.
\end{multline*}
Since $S_{22}$ and $S_{23}$ are treated in exactly the same way we only consider $S_{22}$. 
Writing 
\begin{eqnarray*}
\int_{-\frac{1}{k(k+k_1)}}^{-\frac{1}{k(N+k)}}
= \sum_{l=k_1+k}^{N+k-1} \int_{-\frac{1}{kl}}^{-\frac{1}{k(l+1)}},
\end{eqnarray*}
we see 
\begin{multline} \label{est4}
S_{22} \ll  \left|   \sum_{r=1 }^{\infty} a(r)
  \sum_{c|k} 
\sum_{l=N+1}^{N+k-1} \int_{-\frac{1}{kl}}^{-\frac{1}{k(l+1)}} 
z^{-\frac{1}{2}} \cdot 
e^{-\frac{2 \pi}{kz}\left(r-\frac{1}{24}\right)+\frac{2 \pi z}{k}\left(n-\frac{1}{24} \right)} 
d\Phi \right.  \\ \left.
(-1)^{ak+1} \sin \left(\frac{\pi a}{c} \right)
\sum_{h\atop N < k+k_1 \leq l} \frac{\omega_{h,k}}{\sin\left(\frac{\pi a h'}{c} \right)} \cdot 
e^{- \frac{3 \pi i a^2 k_1h'}{c}} \cdot e^{- \frac{2 \pi i hn}{k} } e^{\frac{2 \pi i m_r h'}{k}} \right|. 
\end{multline}
It follows from the theory of Farey fractions that
\begin{eqnarray*}
k_1 \equiv - h' \pmod k, \quad k_2 \equiv h' \pmod k,  
\\     
N-k<k_1 \leq N, \quad N-k < k_2 \leq N.
\end{eqnarray*}
Thus we can use Lemma \ref{kloost} and estimate (\ref{est4}) as in the case of $S_{21}$. Therefore  $\sum_1$ equals 
\begin{eqnarray*}
 i \sin \left(\frac{\pi a}{c} \right)   \sum_{h,k\atop c|k} 
  \omega_{h,k} \cdot \frac{(-1)^{ak+1}}{
\sin \left(\frac{\pi a h'}{c} \right)} \cdot  e^{- \frac{3 \pi i a^2 k_1h'}{c} - \frac{2 \pi i hn}{k} } 
\int_{- \vartheta_{h,k}'}^{\vartheta_{h,k}^{''}}
z^{-\frac{1}{2}}  \cdot 
e^{\frac{2 \pi z}{k}\left(n-\frac{1}{24}  \right)+\frac{\pi}{12k z }}   \
d\Phi
+ O\left( n^{\epsilon}\right).
\end{eqnarray*}
The sum $\sum_2$ is treated in a similar manner, we make some comments about necessary modifications. One can show that one can write 
\begin{multline} \label{fourex}
e^{-\frac{2 \pi i h' sa}{c} -\frac{3 \pi i h'a^2 k_1}{cc_1}  +\frac{6 \pi i h' l a}{cc_1}   + \frac{\pi}{12k z }}
 \cdot 
  q_1^{\frac{sl}{c_1} -\frac{3 l^2}{2 c_1^2} }  \cdot 
  N \left(ah',\frac{lc}{c_1},c ;q_1\right)   
  =:
  \sum_{r \geq r_0 } b(r) e^{\frac{2 \pi i m_rh'}{k}} e^{-\frac{\pi i r}{12kc^2z}},
  \end{multline}
   where $m_r$ is a sequence in $\Z$, $r_0 \in \Z$. A lengthy but straightforward calculation shows that the terms with $r$ negative, i.e., the the part that gives a contribution for the main term,  can only arise if $s\in \{0,3\}$. If $s=0$, then the contribution is given by
  \begin{eqnarray*}
 \frac{i}{2}  e^{ -\frac{3 \pi i h'a^2 k_1}{cc_1}  +\frac{6 \pi i h' l a}{cc_1}  -\frac{\pi i a h'}{c}+ \frac{\pi}{12k z }}
 \cdot 
  q_1^{-\frac{3 l^2}{2 c_1^2} + \frac{l}{2c_1} }
  \sum_{r \atop \delta_{c,k,r}>0} 
e\left(-\frac{ah'r}{c}\right)  \cdot
q_1^{\frac{lr}{c_1}}.
  \end{eqnarray*}
 If $s=3$, then the contribution is given by
  \begin{eqnarray*}
 \frac{i}{2}  e^{-\frac{6 \pi i a h'a}{c}  -\frac{3 \pi i h'a^2 k_1}{cc_1}  +\frac{6 \pi i h' l a}{cc_1}  + 
 \frac{\pi i a h'}{c}+ \frac{\pi}{12k z }}
 \cdot 
  q_1^{\frac{5l}{2c_1}-\frac{3 l^2}{2 c_1^2} -1} 
  \sum_{r \atop \delta_{c,k,r}>0} 
e\left(\frac{ah'r}{c}\right) \cdot
q_1^{\left( 1-\frac{l}{c_1}\right)r}.
  \end{eqnarray*}
   One can write the sum over $k$ in groups in which the $k$'s have the same values for $c_1$ and $l$ and thus the same values for $\delta_{c,k,r}$, where  $\delta_{c,k,r}$ was defined in (\ref{del}). In each class the condition $\delta_{c,k,r}>0$ is now independent of the $k$ in this class; in every group there are only finitely many terms with $\delta_{c,k,r}>0$, and the number is bounded in terms of $c$. 
   Moreover the coefficient $b(r)$ is independent of $k$ and $a$ in a fixed class.
   Now the terms with negative exponents can be estimated  as before, thus $\sum_2$ equals
  \begin{displaymath}
  2     \sin \left(\frac{\pi a}{c} \right)   \sum_{k,r\atop {c\nmid k  \atop \delta_{c,k,r}>0} } 
  (-1)^{ak+l}  \sum_{h} \omega_{h,k}
  e^{\frac{2 \pi i }{k}(-nh+ m_{a,c,k,r} h')} 
  \int_{- \vartheta_{h,k}'}^{\vartheta_{h,k}^{''}}
z^{-\frac{1}{2}} 
e^{\frac{2 \pi z}{k}\left(n-\frac{1}{24}  \right)+\frac{2 \pi}{k z }  \delta_{c,k,r}} 
d\Phi + O\left(n^{\epsilon} \right).
  \end{displaymath}
  We turn back to the estimation of $\sum_1$. 
  We can write 
  \begin{eqnarray*}
\int_{- \vartheta_{h,k}'}^{\vartheta_{h,k}^{''}}
=
\int_{-\frac{1}{kN}}^{\frac{1}{kN}}  -
 \int_{-\frac{1}{kN}}^{-\frac{1}{k(k+k_1)}} 
- \int_{\frac{1}{k(k+k_2)}}^{\frac{1}{kN}}
\end{eqnarray*}
and denote the associated sums by $S_{11}$, $S_{12}$, and $S_{13}$, respectively.
The sums  $S_{12}$ and $S_{13}$ contribute to the error terms. Since they  have the same shape, we only consider $S_{12}$.
Writing
\begin{eqnarray*}
\int_{-\frac{1}{kN}}^{-\frac{1}{k(k+k_1)}}
= \sum_{l=N}^{k+k_1-1} 
\int_{-\frac{1}{kl}}^{-\frac{1}{k(l+1)}}
\end{eqnarray*}
gives
\begin{multline*}
S_{12} \ll   \sum_{c| k} 
\sum_{l=N}^{N+k-1} \int_{-\frac{1}{kl}}^{-\frac{1}{k(l+1)}} 
z^{-\frac{1}{2}}
e^{\frac{ \pi}{12kz}+\frac{2 \pi z}{k}\left(n-\frac{1}{24} \right)} 
d\Phi  \\
(-1)^{ak+1} \sin \left(\frac{\pi a}{c} \right)
\sum_{h\atop l < k+k_1-1 \leq N+k-1} \frac{\omega_{h,k}}{\sin\left(\frac{\pi a h'}{c} \right)} \cdot 
e^{- \frac{3 \pi i a^2 k_1h'}{c}}  \cdot e^{- \frac{2 \pi i hn}{k} }. 
\end{multline*}
Using $\re(z)=\frac{k}{n}$, $\re\left( \frac{1}{z}\right)<k$, and $|z|^2 \geq \frac{k^2}{n^2}$, this sum can be estimated as before against $O\left(n^{\epsilon} \right)$. Thus, using  (\ref{kl1}),
\begin{eqnarray*}
\sum_1=
 i  \sum_{c|k} 
B_{a,c,k}(-n,0)
\int_{- \frac{1}{kN}}^{\frac{1}{kN}}
z^{-\frac{1}{2}} \cdot 
e^{\frac{2 \pi z}{k}\left(n-\frac{1}{24}  \right)+\frac{\pi}{12k z }}  \
d \Phi
+ O\left( n^{\epsilon}\right).
\end{eqnarray*}
In the same way, we obtain, using  (\ref{kl2}), 
  \begin{eqnarray*}
  \sum_2=
  2    \sin \left(\frac{\pi a}{c} \right)   \sum_{k,r\atop {c\nmid k  \atop \delta_{c,k,r}>0} } 
  D_{a,c,k}(-n,m_{a,c,k,r}) 
  \int_{- \frac{1}{kN}}^{\frac{1}{kN}}
z^{-\frac{1}{2}} \cdot
e^{\frac{2 \pi z}{k}\left(n-\frac{1}{24}  \right)+\frac{2 \pi}{k z }  \delta_{c,k,r}} 
\ d\Phi + O \left(n^{\epsilon} \right).
  \end{eqnarray*}
To finish the estimation of $\sum_1$ and $\sum_{2}$ we have to consider integrals of the form
 \begin{eqnarray*}
I_{k,r}:= \int_{-\frac{1}{kN}}^{\frac{1}{kN}} 
z^{-\frac{1}{2}} \cdot 
e^{\frac{2 \pi}{k}\left(z \left(n - \frac{1}{24} \right) +\frac{r}{z}\right)} d\Phi
.
\end{eqnarray*}  
Substituting $z=\frac{k}{n}- i k \Phi$   gives 
\begin{eqnarray} \label{int}
I_{k,r}= \frac{1}{ki} 
\int_{\frac{k}{n}-\frac{i}{N}}^{\frac{k}{n}+\frac{i}{N}} 
z^{-\frac{1}{2}} \cdot 
e^{\frac{2 \pi}{k}\left(z \left(n - \frac{1}{24} \right) +\frac{r}{z}\right)} \ dz 
.
\end{eqnarray}
We now denote the circle through $\frac{k}{n} \pm \frac{i}{N}$ and tangent to the imaginary axis at $0$ by $\Gamma$. 
If $z=x+iy$, then $\Gamma$ is given by $x^2 + y^2 = \alpha x$, with $\alpha = \frac{k}{n} +\frac{n}{N^2k}$.
Using the fact that $2 > \alpha >\frac{1}{k}$, $\re(z)\leq \frac{k}{n}$, and $\re\left( \frac{1}{z}\right)<k$ on the smaller arc we can show that the integral along the smaller arc is in $O \left(n^{-\frac{3}{4}} \right)$. Moreover the path of integration in (\ref{int}) can be changed by Cauchy's Theorem into the larger arc of $\Gamma$.  Thus 
\begin{eqnarray*}
I_{k,r}= \frac{1}{ki} 
\int_{\Gamma} 
z^{-\frac{1}{2}} \cdot
e^{\frac{2 \pi}{k}\left(z \left(n - \frac{1}{24} \right) +\frac{r}{z}\right)}\ dz 
+ O \left( n^{-\frac{3}{4}}\right).
\end{eqnarray*}
Making the substitution $t = \frac{2 \pi r}{kz}$ gives 
\begin{displaymath}
I_{k,r} 
=
\frac{2 \pi}{k} \left( \frac{2 \pi r}{k} \right)^{\frac{1}{2}}
\frac{1}{2 \pi i } 
\int_{\gamma - i \infty}^{\gamma+ i \infty}
t^{-\frac{3}{2}} \cdot e^{t+\frac{\alpha}{t}} \ dt + O\left(n^{-\frac{3}{4}} \right),
\end{displaymath}
where $\gamma \in \R$ and $\alpha =\frac{ \pi^2 r}{6k^2}(24n-1)$. By the Hankel integral formula we 
now get
\begin{eqnarray*}
I_{k,r}
= 
\frac{4\sqrt{3}}{\sqrt{k   \left(24n - 1 \right)}} \sinh\left(\sqrt{\frac{2r(24n-1)}{3}}\frac{ \pi}{k}
  \right) 
+ 
O \left(n^{-\frac{3}{4}} 
\right).
\end{eqnarray*}
Thus we obtain, estimating the Kloosterman sums with the same arguments as before and using  Lemma \ref{intest},
\begin{multline*}
\sum_1+\sum_2= 
\frac{4 \sqrt{3} i }{  \sqrt{24n-1}} \sum_{c|k} 
\frac{B_{a,c,k}(-n,0)}{\sqrt{k}}  \cdot
\sinh \left(\frac{\pi}{6k} \sqrt{24n-1} \right) \\
+
  \frac{8 \sqrt{3}     \sin \left(\frac{\pi a}{c} \right) }{\sqrt{24n-1}}
    \sum_{k,r\atop {c\nmid k  \atop \delta_{c,k,r}>0} } 
\frac{D_{a,c,k}(-n,m_{a,c,k,r})}{\sqrt{k}}   \cdot
 \sinh \left( 
 \sqrt{\frac{2 \delta_{c,k,r}(24n-1)}{3}}
 \frac{\pi}{k} 
 \right)
 +O \left( n^{\epsilon}\right).
   \end{multline*}
To estimate  $\sum_3$, we split the path of integration as in (\ref{pathint})  and estimate the sums in the same way as before, using that  
$$
\sum_{1 \leq  \nu \leq k} 
g_{a,c,k,\nu}  
 \ll  \sum_{1 \leq \nu \leq 6 c k} \frac{1}{\nu}  \ll
  k^{\epsilon}.
 $$
  \end{proof}
  \begin{proof}[Proof of Corollary \ref{cor1}]
  Corollary \ref{cor1} follows directly from the identity
  \begin{eqnarray} \label{partid}
\sum_{n=0}^{\infty} N(a,c;n)q^n= \frac{1}{c}\sum_{n=0}^{\infty}
p(n)q^n+
\frac{1}{c}\sum_{j=1}^{c-1}
\zeta_c^{-aj}\cdot R(\zeta_c^j;q)
\end{eqnarray}
and  (\ref{Radformula}).
\end{proof}
\section{Proof of Theorem \ref{main2}} \label{ProofTH2}
Here  we prove the Andrews-Lewis conjecture. Using equation  (\ref{partid}) and 
$$
N \left(\frac{1}{3} ;q\right) = N \left(\frac{2}{3};q \right)
$$
easily gives that 
$$
N(0,3;n)< N(1,3;n) \Leftrightarrow 
A\left(\frac{1}{3};n\right)<0.
$$
Thus Theorem 2 follows from  the following.
\begin{proposition}
We have for all $n \not \in \left\{1,3,7\right\} $
\begin{eqnarray}
\label{eq1}
A \left( \frac{1}{3};3n\right)&<&0, \\ 
\label{eq2}
A \left( \frac{1}{3};3n+1\right)&>&0,\\ \label{eq3}
A \left( \frac{1}{3};3n+2\right)&<&0.
\end{eqnarray}
\end{proposition}
\begin{remark}
For $n \in \{1,3,7 \}$, equation (\ref{eq2}) and (\ref{eq3}) still hold and in (\ref{eq1}) we have equality.
\end{remark}
\begin{proof}
With MAPLE we compute that the proposition holds for $n \leq 3000$. Thus in the following, we may assume that $n>3000$.  Actually we only need that $n \geq 300$.

We first consider the main term given by 
\begin{eqnarray} \label{main}
\frac{4 \sqrt{3} i }{(24n-1)^{\frac{1}{2}}}
\sum_{3|k} 
\frac{B_{1,3,k}(-n,0)}{\sqrt{k}} \cdot
\sinh \left( \frac{\pi}{6k}\sqrt{24n-1}\right).
\end{eqnarray}
Using definition (\ref{omega})  it is easy to see that 
$$
B_{1,3,3}(-n,0)
= 2 i \sin \left( \frac{\pi}{18}-\frac{2 \pi n}{3}\right).
$$
Thus the term corresponding to $k=3$ is given by
\begin{eqnarray*}
- \frac{8 \sin\left(\frac{\pi}{18}-\frac{2 \pi n}{3} \right) \cdot \sinh \left(\frac{\pi}{18} \sqrt{24n-1} \right)}{(24n-1)^{\frac{1}{2}}}.
\end{eqnarray*}
Using the trivial bound for Kloosterman sums gives that the remaining terms can be estimated against
\begin{eqnarray*}
\frac{12}{(24n-1)^{\frac{1}{2}} } 
\sum_{2 \leq k \leq \frac{N}{3}} k^{\frac{1}{2}}  \cdot 
\sinh \left( \frac{\pi}{18k}\sqrt{24n-1}\right) .
\end{eqnarray*}
Next we make the $O$-term in Theorem \ref{main1} explicit. For this we use the same notation as in the  proof of Theorem \ref{main1}. Instead of using Lemma \ref{kloost}, we estimate the  Kloosterman sums trivially. The contribution coming from $N\left(\frac{h'}{3};q_1 \right)$ can be estimated against
\begin{eqnarray*}
\frac{2 \cdot e^{2 \pi +\frac{\pi}{24}}}{\sqrt{3}} 
\sum_{r=1}^{\infty}  |a(r)| \cdot e^{-\pi r} 
\sum_{1 \leq k \leq \frac{N}{3}} k^{-\frac{1}{2}}.
\end{eqnarray*}
Computing the first few coefficients and then using that for $n>1$ 
\begin{eqnarray*}
p(n)< \exp \left(\pi \sqrt{\frac{2n}{3}} \right) 
\end{eqnarray*}
easily gives that 
\begin{eqnarray*}
\sum_{r=1}^{\infty}  |a(r)|  \cdot e^{-\pi r}   \leq 0.06.
\end{eqnarray*}
In the same way we can estimate the contribution coming from $N(h',l,3;q_1)$ by
\begin{eqnarray*}
4 \sqrt{3} \cdot 
e^{2 \pi} 
\sum_{r=1}^{\infty}  |b(r)| \cdot e^{- \frac{\pi r}{144} }
\sum_{1 \leq k \leq N\atop 3 \nmid k} k^{-\frac{1}{2}}.
\end{eqnarray*}
Moreover, distinguishing the cases $l=1$ and $l=2$,  we get
\begin{eqnarray*}
\sum_{r=1}^{\infty}  |b(r)| \cdot e^{- \frac{\pi r}{216} } \leq 0.353.
\end{eqnarray*}
By making the path  of integration symmetric, we introduce an error that can be estimated against
\begin{eqnarray*}
2 \sqrt{3} \cdot e^{2 \pi +\frac{\pi}{12}}  \cdot n^{-\frac{1}{2}}
\sum_{1 \leq k \leq \frac{N}{3}} k^{\frac{1}{2}}.
\end{eqnarray*}
Integrating along the smaller arc of $\Gamma$ gives an error that can be estimated against \begin{eqnarray*}
8 \pi \cdot 
e^{2 \pi +\frac{\pi}{24}} \cdot  n^{-\frac{3}{4}}
\sum_{1 \leq k \leq \frac{N}{3}} k.
\end{eqnarray*}
Moreover the estimate in Lemma \ref{intest} can be made explicit as 
\begin{eqnarray*}
2^{\frac{5}{4}} \cdot \frac{e+e^{-1}}{3} \cdot n^{\frac{1}{4}} 
\left(\min\left(\left\{\frac{\nu}{k} -\frac{1}{6k}+\frac{1}{3}\right\},  \left\{\frac{\nu}{k} -\frac{1}{6k}+\frac{1}{3}\right\} \right)\right)^{-1}.
\end{eqnarray*}
This easily gives 
\begin{eqnarray*}
\sum_3 
\leq  2^{\frac{1}{4}} \cdot (e+e^{-1}) \cdot e^{2 \pi } \cdot n^{-\frac{1}{4}}
\sum_{k} \frac{1}{k}
 \sum_{\nu=1}^{k} 
\left( \min \left(\left\{\frac{\nu}{k} -\frac{1}{6k}+\frac{1}{3}\right\}, \left\{ \frac{\nu}{k} -\frac{1}{6k}-\frac{1}{3}\right\} \right)\right)^{-1}.
\end{eqnarray*}
Combining all errors one can show that  the term in (\ref{main}) is dominant for $n \geq 300$.
\end{proof}


\begin{thebibliography}{99}

\bibitem{An1} G. E. Andrews, \emph{The theory of partitions},
Cambridge Univ. Press, Cambridge, 1998.

\bibitem{An2} G. E. Andrews, \emph{On the theorems of
Watson and Dragonette for Ramanujan's mock theta functions}, Amer.
J. Math.  \textbf{88} No. 2 (1966), pages 454-490.

\bibitem{An3} G. E. Andrews, \emph{Mock theta functions,}
Theta functions - Bowdoin 1987, Part 2 (Brunswick, ME., 1987),
pages 283-297, Proc. Sympos. Pure Math. \textbf{49}, Part 2, Amer.
Math. Soc., Providence, RI., 1989.

\bibitem{AL} G. E. Andrews and R. P. Lewis,
\emph{The ranks and cranks of partitions moduli 2, 3 and 4}, J.
Number Th. \textbf{85} (2000), pages 74-84.

\bibitem{ASD} A. O. L. Atkin and H. P. F. Swinnerton-Dyer,
\emph{Some properties of partitions}, Proc. London Math. Soc.
\textbf{66} No. 4 (1954), pages 84-106.

\bibitem{BO1} K. Bringmann and K. Ono,
\emph{The $f(q)$ mock theta function conjecture and partition
ranks}, Invent. Math.
\textbf{165} (2006), pages 243-266.

\bibitem{BO2} K. Bringmann and K. Ono, \emph{Dysons ranks and Maass forms},  Annals, accepted for publication.   




\bibitem{Dr} L. Dragonette,
\emph{Some asymptotic formulae for the mock theta series of
Ramanujan}, Trans. Amer. Math. Soc.
\textbf{72} No. 3 (1952), pages 474-500.


\bibitem{Dy1} F. Dyson, \emph{Some guesses in the theory of
partitions}, Eureka (Cambridge) \textbf{8} (1944),
pages 10-15.

\bibitem{Dy2} F. Dyson,
\emph{A walk through Ramanujan's garden}, Ramanujan revisited
(Urbana-Champaign, Ill. 1987), Academic Press, Boston, 1988, pages
7-28.

\bibitem{Ga} F. Garvan: Generalization of Dyson's rank and non-Rogers-Ramanujan partitions.
Manuscripta Math. \textbf{84} (1994), pages 343-359.

\bibitem{GM} B. Gordon and R. McIntosh,
\emph{Modular transformations of Ramanujan's fifth and seventh
order mock theta functions}, Ramanujan J. \textbf{7} (2003), pages
193-222.

\bibitem{Le} R. P. Lewis, \emph{The ranks of partitions modulo
2}, Discuss. Math. \textbf{167/168} (1997), pages 445-449.


\bibitem{Ra} S. Ramanujan,
\emph{The lost notebook and other unpublished papers}, Narosa, New
Delhi, 1988.



\end{thebibliography}
\end{document}